\documentclass[a4paper,12pt,intlimits,oneside]{amsart}
\usepackage{amsmath,amsthm,amssymb}
\usepackage{dsfont}
\usepackage{mathrsfs}
\usepackage[utf8]{inputenc}
\usepackage[style=alphabetic,maxnames=200,maxalphanames=6,giveninits]{biblatex}
\usepackage[colorlinks=true]{hyperref}

\usepackage{esint}
\usepackage{amsthm}
\usepackage{latexsym}
\usepackage{amssymb}
\usepackage{xcolor}
\usepackage{hyperref}
\usepackage{bbold}
\usepackage{enumerate}

\makeatletter
\numberwithin{equation}{section}
\theoremstyle{plain}
        \newtheorem{theorem}{Theorem}[section]
        \newtheorem{proposition}[theorem]{Proposition}
        \newtheorem{lemma}[theorem]{Lemma}
        \newtheorem{corollary}[theorem]{Corollary} 
        
        \newtheorem{definition}[theorem]{Definition} 
        \newtheorem{remark}[theorem]{Remark}  
         
        \newtheorem*{claim*}{Claim}

\newtheorem*{theorem*}{Theorem}
\newtheorem*{definition*}{Definition}
\newtheorem*{proposition*}{Proposition}
\usepackage{graphicx}
\let\oldmarginpar\marginpar
\renewcommand\marginpar[1]{\-\oldmarginpar[\raggedleft\footnotesize #1]
{\raggedright\footnotesize #1}}
\newcommand \be  {\begin{equation}}
\newcommand \ee  {\end{equation}}
\newcommand \la \langle
\newcommand \ra \rangle

\newcommand \lan {\langle}
\newcommand \ran {\rangle}

\newcommand{\R}{\mathbb{R}}
\newcommand{\C}{\mathbb{C}}
\newcommand{\Q}{\mathbb{Q}}
\newcommand{\T}{\mathbb{T}}
\newcommand{\Z}{\mathbb{Z}}

\newcommand{\N}{\mathbb{N}}

\newcommand\e{\varepsilon}

\renewcommand{\Re}{\mathrm{Re}}
\renewcommand{\Im}{\mathrm{Im}}

\def\build#1_#2^#3{\mathrel{\mathop{\kern 0pt#1}\limits_{#2}^{#3}}}
\def\td_#1,#2{\mathrel{\mathop{\build\longrightarrow_{#1\rightarrow #2}^{}}}}

\title[Soliton and breather resolution for the Szeg\H{o} equation]{Soliton and breather resolution\\ for the cubic Szeg\H{o} flow on the line}

\author[P. Gérard]{Patrick G\'erard}
\author[S. Grellier]{Sandrine Grellier}
\address[P. Gérard]{Laboratoire de Math\'ematiques d'Orsay, Universit\'e Paris-Saclay,  91405 Orsay, France, CNRS}
\email{patrick.gerard@universite-paris-saclay.fr}
\address[S. Grellier]{Institut Denis Poisson, D\'epartement de Math\'ematiques, Universit\'e d'Orl\'eans, 45067 Orl\'eans Cedex 2, France}
\email{sandrine.grellier@univ-orleans.fr}

\begin{document}
\date{\today}

\begin{abstract}
We investigate the long--time behaviour of the solutions of the cubic Szeg\H{o} equation on the line in the Sobolev space $H^{1/2}(\R)$. We prove that, for every datum of which the Lax operator has simple positive spectrum,  the solution asymptotically decouples as an infinite sum of traveling quasi--periodic breather solutions. Under an additional generic condition on the data, we prove that these traveling breather solutions are in fact soliton solutions, leading to a soliton resolution theorem. 
\end{abstract}
\maketitle
\section{Introduction}
\subsection{The problem of long--time dynamics for the cubic Szeg\H{o} equation.}

The cubic Szeg\H{o} equation has been introduced by the authors in \cite{GG10} as a model of non--dispersive Hamiltonian evolution on the torus. Its Lax pair structure led to a thorough study of its dynamics \cite{GG15}, \cite{GG17}, \cite{GPInv}, including an explicit formula for the solutions, almost periodicity of the trajectories on the Sobolev space $H^{1/2}$, generic long--time growth of  Sobolev norms $H^s$ for $s>1/2$, and a continuous extension of the flow map to any square integrable data.

On the line, the cubic Szeg\H{o} equation was introduced in \cite{Poc11T}, \cite{Poc11}, where the Lax pair structure was discovered along the lines of \cite{GG10}, where  traveling wave solutions were characterized and proved to be stable, and where the study of the long--time dynamics was initiated. In particular, a soliton resolution theorem was proved for a generic class of rational initial data, while an example was given of a rational initial datum generating growth of $H^s$ Sobolev norms for $s>\frac 12$, in connection with the existence of a multiple eigenvalue for the Lax operator. This  phenomenon of long--time transition to high frequencies was proved to be generic in \cite{GPJEMS}, in connection with the multiplicity of the Lax spectrum.
Finally, an explicit formula for arbitrary solutions and extension of the flow map to square integrable data was proved in \cite{GPCMP}.

On the other hand, recent soliton resolution results were established \cite{GL25}, \cite{GaGM26}, \cite{GaG26} for integrable equations enjoying the same kind of explicit formula as the one in \cite{GPCMP} for the cubic Szeg\H{o} equation on the line. The purpose of this paper is to investigate such a long--time behaviour for the cubic Szeg\H{o} equation on the line, in the light of these results and of the formula discovered in \cite{GPCMP}. Because of the above--mentioned growth of high Sobolev norms and instability generated by multiple Lax positive spectrum, we shall restrict ourselves to long--time behavior in $H^{\frac 12}$ of solutions with simple Lax positive spectrum. In that framework, we prove indeed that any solution can be decomposed as a sum of building--block simple solutions. However, the new phenomenon that we want to highlight is that these building block solutions are not in general solitons, but traveling quasi--periodic breathers. The occurence of breathers in long-time soliton resolutions for the sine-Gordon equation has been recently highlighted in \cite{CLL}. In order to tell more about this phenomenon, we need to introduce the setting of the equation.

\subsection{The setting of the cubic Szeg\H{o} equation on the line}
In the following, $L^2_+(\R)$ denotes the Hardy space of the upper half plane that is the set of functions in $L^2(\R)$ whose Fourier transform is supported on $(0,\infty)$. Here the Fourier transform is given for $f\in L^2(\R)$ by 
$$\hat f(\xi)=\int_\R f(x){\rm e}^{-ix\xi} dx.$$
According to the Paley--Wiener theorem, such functions $f\in L^2_+(\R)$ admit holomorphic extensions to the upper half plane that we still denote by $f$, and given by
$$f(z)=\frac 1{2\pi}\int_0^\infty \hat f(\xi){\rm e}^{iz\xi} d\xi,\; z\in \C_+.$$
In this work, we are interested in the asymptotic behaviour of solutions to the cubic Szeg\H{o} equation
\begin{equation}\label{Szego}\left\{\begin{array}{lcl}i\partial _tu&=&\Pi(|u|^2u),\; u=u(t,x),\; x,t\in\R\\
u(0,x)&=&u_0(x)\end{array}\right.
\end{equation}
where $u\in C(\R, H^{\frac 12}_+(\R))$, where, for every $s\geq 0$, 
$$H^{s}_+(\R):=\{f\in L^2_+(\R),\;\xi \mapsto \xi^{ s}\hat f(\xi)\in L^2_+([0,\infty[)\}$$ equipped with the norm
$$\Vert f\Vert_{H^s}^2:=\frac 1{2\pi}\int_\R (1+\xi^2)^s|\hat f|^2(\xi) d\xi.$$
The global wellposedness of \eqref{Szego} in $H^s_+(\R)$ for every $s\geq \frac 12$ has been established in \cite{Poc11T}. Moreover, it has been shown that \eqref{Szego} enjoys a Lax pair structure, which we briefly describe below.

We recall that the Hankel operator of symbol $u$ is defined on $L^2_+$ as 
$H_u(f)=\Pi(u\overline f)$, when well defined, where $\Pi$ is the orthogonal projector from $L^2$ into $L^2_+$, the so-called Szeg\H{o} projector. Equivalently, one has 
$$\forall f\in L^2(\R),\quad \Pi (f)(x)=\frac1{2\pi}\int_0^\infty {\rm e}^{ix\xi}\hat f(\xi)d\xi$$ and
$$\forall f\in L^2_+(\R),\quad \widehat{H_u(f)}(\xi)=\frac 1{2\pi}\int_0^\infty \hat u(\xi+\eta)\overline{\hat f(\eta)} d\eta.$$ It is well known that a Hankel operator $H_u$ is a Hilbert-Schmidt  operator on $L^2_+(\R)$ whenever $u\in H^{\frac 12}_+(\R)$. 
Moreover, computing the trace of $H_u^2$, one gets 
\begin{equation}\label{TrH2}
Tr(H_u^2)=\frac 1{(2\pi)^2}\int_0^\infty \xi|\hat u|^2(\xi) d\xi=:\frac 1{2\pi}\Vert u\Vert_{\dot H^{\frac 12}}^2.
\end{equation}
Here $\dot H^{\frac 12}$ denotes the homogeneous Sobolev space.
With this notation, for every $u_0\in H^s_+(\R)$ with $s>\frac 12$, the solution $u\in C(\R, H^s_+(\R))$  \eqref{Szego} satisfies
\begin{equation}\label{Laxpair}
\frac{ dH_u}{dt}=[B_u,H_u],
\end{equation}
where $B_u={-iT_{|u|^2}+\frac i2 H_u^2}.$ Here $T_b$ is the Toeplitz operator defined on $L^2_+(\R)$  by $T_b(f)=\Pi(bf).$ 
The Lax pair formula \eqref{Laxpair} implies that $H_{u(t)}$ remains unitarily equivalent to $H_{u_0}$. Indeed,  one may define the operator-valued function $U(t)$ solution of the system 
\begin{equation}\label{Unitary}
\left\{\begin{array}{lcl}\frac {d}{dt}U(t)&=&B_{u(t)}U(t),\; \\
U(0)&=&\mathrm{Id}\end{array}\right.
\end{equation}
Then, one can prove that $U(t)$ is unitary and that
\begin{equation}\label{UHu}
U^*(t)H_{u(t)} U(t)=H_{u_0}.
\end{equation}
This property implies that the spectrum of the self--adjoint positive trace class operator $H_{u(t)}^2$ is independent of $t$. By an easy approximation argument, the latter can be extended to any datum $u_0\in H^{\frac 12}_+$. Moreover, identity \eqref{UHu} plays a fundamental role to obtain an explicit formula for the solution to \eqref{Szego} proved in \cite{GPCMP}. We will recall and extend this formula in the next section.\\

In the work \cite{Poc11T}, the author has characterized the traveling waves for the cubic Szeg\H{o} equation on the line and established their stability.
 Traveling waves correspond to special rational functions of the form 
 \begin{equation}\label{TravelingWaves}
 R_{a,p}(x)=\frac{a}{x+p},\; a\in\C, \; p\in\C_+
 \end{equation} with speed $c:=\frac{|a|^2}{2\Im p}$ and frequency $\omega:=\frac{|a|^2}{4(\Im p)^2}.$
 In other words, the functions \begin{equation}\label{soliton}
 (t,x)\mapsto {\rm e}^{-i\omega t}R_{a,p}(x-ct)
 \end{equation} solve the cubic Szeg\H{o} equation and these are the only possible traveling waves.
At this stage, let us introduce a larger class of special solutions of \eqref{Szego}.

\begin{definition}\label{breather}
Let $N$ be a positive integer and $c$ be a real number. A solution $R\in C(\R,H^{\frac 12}_+(\R))$ of the cubic Szeg\H{o} equation  \eqref{Szego} is called a traveling quasi--periodic  breather of order $N$ and of velocity $c$, if there exist  a vecteur $\omega \in \R^N$ and a continuous function $G:\T^N\to H^{\frac 12}_+(\R)$ such that
$${R(t)(x)=G(t\omega)(x-ct)}.$$
\end{definition}
In particular, traveling waves are traveling (quasi)--periodic breathers of order $1$. The notion of breather is identified in the literature of Hamiltonian PDEs, though it is known to represent rather rare solutions in the context of dispersive equations, see e.g. \cite{CaW24} and references therein. In the case of the non--dispersive equation \eqref{Szego}, it will turn out to be a building block of the dynamics. Notice that examples of traveling quasi--periodic  breathers of any order $N$ can be built {as rational functions of $x$ } using the solution of the inverse spectral problem in \cite{GPJEMS}.
\subsection{Statement of the results} Our main result can be stated as follows.
\begin{theorem}\label{maintheorem}
Let $u_0\in H^{\frac 12}_+(\R)$ be such that positive eigenvalues of $H_{u_0}^2$ are all simple. {Then there exists a countable set $\mathcal V$ of positive numbers,
a family $(N_c)_{c \in \mathcal V}$ of positive integers such that $\sum_{c \in \mathcal V} cN_c <\infty$, and a family $(R_c)_{c\in \mathcal V} $ of traveling quasi-periodic breather solutions of \eqref{Szego} of order $N_c $ and of velocity $c $ such that
\[\lim_{t\to \pm \infty}\left \Vert u(t)-\sum_{c \geq \e}R_c (t)\right \Vert _{H^{\frac 12}}\td_\e, {0^+} 0.\]
}\end{theorem}

In the next sections, we will see how the functions {$R_c$} are defined in terms of the initial datum $u_0$ and the spectral theory of $H_{u_0}^2$, see Theorem \ref{Main2} below. In particular, we will prove that $R_{{c}} (t,\cdot )$ is in fact a rational function with a denominator of degree $N_{{c}}$.

 Let us consider a special case of Theorem \ref{maintheorem} , related to the following definition.

\begin{definition}
A function $u\in H^{\frac 12}_+(\R)$  is said strongly generic if all the non zero eigenvalues of the trace-class operator $H_u^2$ are simple and if, denoting by  $(\varphi_j)$  an orthonormal system of eigenfunctions of $H_u^2$ corresponding to these eigenvalues,  $|\langle u,\varphi_j\rangle| \neq |\langle u,\varphi_k\rangle|$ for any $j\neq k$. 
\end{definition} 
\begin{remark} One can easily check that the set of strongly generic functions is a dense $G_\delta$ subset in $H^{\frac 12}_+$. The terminology "strongly generic" has been introduced in \cite{Poc11} where a soliton resolution result was proved for strongly generic rational solutions. We generalize this result in the following corollary. 
\end{remark}
\begin{corollary}\label{solitonres}
Assume $u_0$ is strongly generic in $H^{\frac 12}_+(\R)$. Then the breathers $R_\mu$ in Theorem \ref{maintheorem} are traveling waves, so that
there exist $(p_j) \in \C_+^\N$, $(a_j)\in\C^\N$  such that 
the Szeg\H{o} solution $u$ associated to $u_0$ satisfies 
$$\lim_{t\to \pm \infty} \left \Vert u(t,\cdot)-\sum_{j=1}^N {\rm e}^{-it\omega_j}R_{a_j, p_j}(\cdot -c_jt)\right \Vert_{H^{\frac 12}_+}\td_N,\infty 0.$$
Here $\omega_j=\frac{|a_j|^2}{4(\Im p_j)^2}$ and $c_j=\frac{|a_j|^2}{2\Im p_j}$.
\end{corollary}
\subsection{Organization of the paper} Let us complete this introduction by describing shortly the structure of the paper. In the next section, we revisit the explicit formula proved in \cite{GPCMP}. Then we prove, as  a fundamental preliminary result, that the eigenfunctions of the Lax operator $H_u^2$ associated to positive eigenvalues belong to the domain of the operator $X^*$ arising in the explicit formula. From this we infer that $u$ is never orthogonal to the eigenspaces of $H_u^2$ associated to positive eigenvalues.
Finally, we prove Theorem \ref{maintheorem} and Corollary \ref{solitonres}. The crucial step is to use the explicit formula for proving that, for every $\mu \in \mathcal M$, 
there exists a quasi--periodic function $g_\mu $ of $t$, valued in rational functions of $L^2_+$, such that $u(t,x+c_\mu t)-g_\mu (t,x)$ converges weakly to $0$ in $H^{\frac 12}_+(\R)$. Then we use the quasi--periodicity of $g_\mu$ to prove that $R_\mu (t,x):=g_\mu (t,x-c_\mu t)$ is a solution of the cubic Szeg\H{o} equation.
The last step is to prove strong convergence in $H^{\frac 12}_+(\R)$, which requires an additional precise calculation of each term in the energy balance. 

\section{Explicit formula}
In this section, we recall the basic definitions and state the explicit formula proved in \cite{GPCMP}.\\
The operator $X$ is the unbounded operator of multiplication by $x$ hence
$${\mathrm {Dom}}\,X:=\{f\in L^2_+(\R);\; xf\in L^2(\R)\}$$
and 
$$x\in\R,\; Xf(x)=xf(x), \; f\in {\mathrm {Dom}}\,X.$$
One has also 
$${\mathrm {Dom}}\,X=\left\{f\in L^2_+(\R),\; \frac d{d\xi}\hat f\in L^2(0,\infty), \;\hat f(0)=0\right\}.$$  

We will use its adjoint $X^*$. By definition 
$${\mathrm {Dom}}\,X^*:=\left\{f\in L^2_+,\; \exists g\in L^2_+, \langle g,h\rangle =\langle f,X h \rangle,\; \forall h\in \mathrm{Dom}X\right\}.$$ With these notations, one has $ X^*f=g$ for $f\in \mathrm{Dom} X^*$.\\
On the Fourier side, it gives
$${\mathrm {Dom}}\,X^*=\left\{f\in L^2_+,\; \frac d{d\xi}\hat f\in L^2(0,\infty)\right\}$$
$$\xi>0,\;\widehat{X^*f}(\xi)=i\frac d{d\xi}\hat f(\xi),\; \; f\in {\mathrm {Dom}}\,X^*$$
so that, 
\begin{equation}\label{DomcotX}
{\mathrm {Dom}}\,X^*:=\left\{f\in L^2_+,\; \exists \lambda_f\in\C, \; xf+\lambda_f\in L^2(\R)\right\}\end{equation} and $X^*f=Xf+\lambda_f$. As $\hat f$ is right continuous at $0$, $\lambda_f=\frac 1{2i\pi} I_+(f)$ a posteriori.\\
Here $I_+(f)=\hat f(0^+):=\lim_{\varepsilon\to 0}\langle f,\chi_\varepsilon\rangle$, $\chi_{\varepsilon}(x)=\frac 1{1-i\varepsilon x}$ when this limit exists.

The operator $iX$ has infinitesimal generator given by the Lax-Beurling semigroup $S(\eta)$ defined on $L^2_+(\R)$ by $S(\eta)f(z):={\rm e}^{i\eta z}f(z)$, $\eta>0$, $z\in\C_+$. In other words, for $f\in {\mathrm {Dom}}\,X$, $$iXf=\lim_{\eta\to 0}\frac{S(\eta)f-f}{\eta }.$$ 
The infinitesimal operator related to $-iX^*$ is $S^*(\eta)f={\rm e}^{-i\eta z} f(z)$ and $$-iX^*f(z):=\lim_{\eta\to 0}\frac{S^*(\eta)f(z)-f(z)}{\eta}.$$
This leads to a generalization of the Cauchy formula.
\begin{proposition}\label{Cauchy}
For every $f\in L^2_+(\R)$, every $z\in \C_+$, 
$$(X^*-z{\rm Id})^{-1}f(x)=\frac{f(x)-f(z)}{x-z},\; f(z)=\frac 1{2i\pi}I_+((X^*-z{\rm Id})^{-1}f).$$
\end{proposition}
\begin{proof} We give the proof for completeness.
For a fixed $z\in\C_+$, denote  by $g_z$ the function defined by
$$g_z(x):=\frac{f(x)-f(z)}{x-z}.$$
From the definition of $X$, we have  ,
$$(X-z{\rm Id})g_z+f(z)=f.$$
Obviously, $g_z\in L^2_+(\R)$, and 
$$(x-z)g_z(x)+f(z)=f(x).$$
We deduce from this equality and from \eqref{DomcotX} that $g_z\in \mathrm{Dom}(X^*)$, that $(X^*-z \mathrm{Id})g_z=f$ and $f(z)=\frac 1{2i\pi}I_+(g_z).$
\end{proof}

Using this last equality and intertwining the unitary operator $U(t)$ \eqref{Unitary} driving the Szeg\H{o} flow leads to the explicit formula obtained by P. Gérard and S. Pushnitski (see \cite{GPCMP}). \begin{theorem}[\cite{GPCMP}]

Let $z\in\C_+$ and $u_0\in H^{\frac 12}_+(\R)$. Then the solution of \eqref{Szego} is given by
\begin{equation}\label{Explicit}
\forall z\in \C_+,\quad u(t,z)=\frac 1{2i\pi} I_+((X^*+L_{u_0}(t)-z)^{-1} {\rm e}^{-itH_{u_0}^2} u_0).
\end{equation}
Here
$$L_{u_0}(t)=\frac 1{2\pi} \int _0^t \langle \cdot, {\rm e}^{-isH_{u_0}^2} u_0\rangle {\rm e}^{-isH_{u_0}^2} u_0 ds.$$
\end{theorem}
We will need a more general version of this formula.
\begin{corollary}\label{cor:explicit}
Let $u_0\in H^{s}_+$, $s>\frac 12$. For $f\in L^2_+(\R)$ define 
\begin{equation}\label{Omega}
\forall z\in \C_+,\quad  \Omega_t(f)(z):=\frac 1{2i\pi} I_+((X^*+L_{u_0}(t)-z)^{-1} f),
\end{equation}
then 
\begin{equation}\label{UOmega}
\Omega_t(f)=U(t){\rm e}^{i\frac t2H_{u_0}^2} f.
\end{equation} In particular, $\Omega_t$ is an isometry on $L^2_+(\R)$. In the general case $u_0\in H^{\frac 12}_+$, one has
$$\Vert\Omega_t f\Vert_{L^2}\le \Vert f\Vert_{L^2}.$$
In particular, for $z\in\C_+$, 
\begin{equation}\label{Pointwise}\left|\Omega_t f(z)\right|\le \frac{\Vert f\Vert_2}{2\sqrt{\pi\mathrm{Im}z}}.\end{equation}
  \end{corollary}
  \begin{proof}
  The proof is slightly more general than the proof of the explicit formula proved in \cite{GPCMP} but follows the same lines. In  \cite{GPCMP} the following equations are established (the notations are the same except that operator $A$ is $X$ here)
  \begin{equation}\label{UstarXstarU}U^*(t)X^*U(t)={\rm e}^{i\frac t2 H_{u_0}^2}(X^*+L_{u_0}(t)){\rm e}^{-i\frac t2 H_{u_0}^2}\end{equation} 
   and  \begin{equation}\label{Ustarchi}U^*(t)\chi_\varepsilon- {\rm e}^{i\frac t2 H_{u_0}^2}\chi_\varepsilon\td_\varepsilon,0 0.\end{equation}

   Hence, for any function $g\in L^2_+$, from Proposition \ref{Cauchy}, we write
  \begin{eqnarray*}
 g(z)&=&\frac 1{2i\pi}I_+((X^*-z)^{-1}g)\\
  &=&\frac 1{2i\pi} \lim_{\varepsilon\to 0}\langle (X^*-z)^{-1}g,\chi_{\varepsilon}\rangle\\
  \end{eqnarray*}
  Introducing the unitary operator $U(t)$ and using  \eqref{UstarXstarU}, we get
  \begin{eqnarray*}
 g(z) &=&\frac 1{2i\pi} \lim_{\varepsilon\to 0}\langle (U^*(t)X^*U(t)-z)^{-1}U^*(t)g,U^*(t)\chi_{\varepsilon}\rangle\\
  &=&\frac 1{2i\pi} \lim_{\varepsilon\to 0}\langle ({\rm e}^{i\frac t2 H_{u_0}^2}(X^*+L_{u_0}(t)){\rm e}^{-i\frac t2 H_{u_0}^2}-z)^{-1}U^*(t)g,U^*(t)\chi_{\varepsilon}\rangle\\
   &=&\frac 1{2i\pi} \lim_{\varepsilon\to 0}\langle ((X^*+L_{u_0}(t))-z)^{-1}{\rm e}^{-i\frac t2 H_{u_0}^2}U^*(t)g,{\rm e}^{-i\frac t2 H_{u_0}^2}U^*(t)\chi_{\varepsilon}\rangle\\
   &=&\Omega_t({\rm e}^{-i\frac t2 H_{u_0}^2}U^*(t)g)(z)
  \end{eqnarray*}
  The last equality is obtained thanks to Equation \eqref{Ustarchi}.
  It gives the result for $f={\rm e}^{-i\frac t2 H_{u_0}^2}U^*(t)g$ or $g=U(t){\rm e}^{i\frac t2 H_{u_0}^2}f$.
    As $U(t)$ and ${\rm e}^{i\frac t2 H_{u_0}^2}$ are isometries, $\Omega_t$ is an isometry on $L^2$ for $u_0\in H^{s}_+$, $s>\frac 12$. 
\\
  When $u_0\in H^{\frac 12}_+$, one obtains the boundedness of $\Omega_t$ using an approximation of $u_0$ by a sequence $u_0^\delta \in H^1_+$.
The pointwise estimate \eqref{Pointwise} follows from the inverse Fourier formula and the Cauchy-Schwarz inequality. Namely, write
$$\left|\Omega_tf(z)\right|\le \frac{1}{2\pi}\int_0^\infty \left|\mathrm e^{iz\xi}\widehat{\Omega_tf}(\xi)\right| d\xi\le  \frac{1}{2\pi}\frac 1{\sqrt{2\mathrm{Im}z}}\Vert \widehat{\Omega_tf}\Vert_{L^2}\le \frac{\Vert f\Vert_2}{2\sqrt{\pi\mathrm{Im}z}}.$$
   \end{proof}

 \section{Preliminary results}
 
In this section, we give some preliminary results which can be found essentially in the literature but we prove them for completeness.

   We first establish the following lemma.
    \begin{lemma}\label{Xstarff}
 Let $f\in {\rm Dom}X^*$ then 
 $$\Im \langle X^*f,f\rangle=-\frac 1{4\pi}|I_+(f)|^2.$$
 \end{lemma}
 \begin{proof}
 By Plancherel formula, 
 $$ \langle X^*f,f\rangle= \frac 1{2\pi}\langle \widehat{X^*f},\hat f\rangle=\frac 1{2\pi}\langle i\frac d{d\xi}\hat f,\hat f\rangle.$$
 Hence, taking the imaginary part gives
$$ \Im \langle X^*f,f\rangle=\frac 1{2\pi}\Re \langle \frac d{d\xi}\hat f,\hat f\rangle=\frac 1{4\pi}\int_0^\infty\frac d{d\xi}|\hat f|^2(\xi) d\xi=-\frac 1{4\pi}|\hat f(0^+)|^2.$$

 \end{proof}
We need another auxiliary lemma which uses the fundamental property on Hankel operator $H_u$.
\begin{equation}\label{SstarH}
S^*(\eta)H_u=H_u S(\eta),\; \eta>0.
\end{equation}
  \begin{lemma}\label{Commutator}
  Let $u\in H^{\frac 12}_+(\R)$. Let $f\in L^2_+(\R)$ such that $\hat f$ is continuous in a neighborhood of $0^+$ (or in other words, $f\in {\rm Dom}I_+$). Then,  
  $$\left[\frac{S(\eta)^*}\eta,H_u^2\right ]f\td_\eta,0\frac 1{2\pi}\left(I_+(f)H_u(u)-\langle f,u\rangle u\right)\; \text{ in }L^2_+(\R).$$
  \end{lemma}
  
  \begin{proof}
 We write 

 $$\left[\frac{S(\eta)^*}\eta,H_u^2\right] =\left[\frac{S(\eta)^*}\eta,H_u\right]H_u+H_u \left[\frac{S(\eta)^*}\eta,H_u\right ]$$
 and 
$$ \left[\frac{S(\eta)^*}\eta,H_u\right] =\frac{S(\eta)^*}\eta H_u-H_u\frac{S(\eta)^*}\eta= H_u\frac{S(\eta)}\eta-H_u\frac{S(\eta)^*}\eta$$
so that, using both identities, we get
 \begin{equation}\label{1}\left[\frac{S(\eta)^*}\eta,H_u^2\right] =H_u\left(\frac{S(\eta)}\eta H_u-H_u \frac{S(\eta)^*}\eta\right )\end{equation}
 Let $g_1:=\frac{S(\eta)}\eta H_uf$ and $g_2:=H_u \frac{S(\eta)^*}\eta f$ and compute $\hat g_1-\hat g_2$. For $\xi>0$, 
  \begin{eqnarray*}\widehat{g_2}(\xi)&=&\frac 1{2\pi \eta} \int_0^\infty \hat u(\xi+\zeta)\overline{\hat f(\eta+\zeta)}d\zeta\\
  &=&\frac 1{2\pi \eta} \int_\eta^\infty \hat u(\xi+\zeta-\eta)\overline{\hat f(\zeta)}d\zeta\\
  &=&\frac 1{2\pi \eta} \int_0^\infty \dots-\frac 1{2\pi \eta} \int_0^\eta \dots.
   \end{eqnarray*}
   Observe that, if $\xi\ge \eta$, 
   $$\frac 1{2\pi \eta} \int_0^\infty \hat u(\xi+\zeta-\eta)\overline{\hat f(\zeta)}d\zeta=\frac{1}{\eta}\widehat{H_uf}(\xi-\eta)=\hat g_1(\xi),$$
 so that 
 \begin{eqnarray*}
  \hat g_2(\xi)&=&-\frac 1{2\pi \eta} \int_0^\eta \hat u(\xi+\zeta-\eta)\overline{\hat f(\zeta)}d\zeta\\
  &+&\left\{\begin{array}{ll}\hat g_1(\xi)&\text{ if }\xi\ge \eta\\
\displaystyle \frac 1{2\pi \eta} \int_0^\infty \hat u(\xi+\zeta-\eta)\overline{\hat f(\zeta)}d\zeta &\text{ if } 0<\xi<\eta.
 \end{array}\right.
     \end{eqnarray*}

It gives
   \begin{eqnarray*}
  \hat g_1(\xi)-\hat g_2(\xi)&=&-\frac 1{2\pi \eta} \int_0^\infty \hat u(\xi+\zeta-\eta)\overline{\hat f(\zeta)}d\zeta\mathds{1}_{[ 0,\eta]}(\xi)\\
   &+&\frac 1{2\pi \eta} \int_0^\eta \hat u(\xi+\zeta-\eta)\overline{\hat f(\zeta)}d\zeta\mathds 1_{[0,\infty]}(\xi)\\
  \end{eqnarray*}

  Applying $H_u$  gives
   \begin{eqnarray*}
  \widehat{H_u\frac{S(\eta)}\eta H_uf}-\widehat{H_u^2\frac{S(\eta)^*}\eta} f&=& -\frac 1{2\pi}\int_0^\infty \hat u(\cdot+\xi) \frac 1{2\pi\eta}{\mathds 1}_{[0,\eta]}(\xi) \int_0^\infty \overline{\hat u(\xi+\zeta-\eta)}\hat f(\zeta)d\zeta d\xi\\
  &+&\frac 1{2\pi}\int_0^\infty \hat u(\cdot+\xi) \frac 1{2\pi\eta} \int_0^\eta\overline{ \hat u(\xi+\zeta-\eta)}\hat f(\zeta)d\zeta d\xi.
  \end{eqnarray*}
The first term equals
$$\frac {-1}{(2\pi)^2 \eta}\int_0^\eta \hat u(\cdot+\xi) \left( \int_0^\infty \overline{ \hat u(\xi+\zeta-\eta)}{\hat f(\zeta)}d\zeta\right)d\xi=\frac {-1}{(2\pi)^2 \eta}\int_0^\eta \hat u(\cdot+\xi)\langle \hat f,\hat u(\xi-\eta+\cdot)\rangle d\xi.$$  As translations are continuous on $L^2$, this term tends to $-\frac 1{2\pi} \hat u (\cdot)\langle f,u\rangle$ in $L^2_+$ as $\eta\to 0$ so that the corresponding term in $  \left[\frac{S(\eta)^*}\eta,H_u^2\right] f$ tends to $-\langle f,u\rangle \frac u{2\pi}.$
 The same argument allows to get that the second term tends in $L^2_+$ to $$\frac 1{2\pi}\int_0^\infty \hat u(\cdot+\xi)\left(\frac 1{2\pi}\overline{\hat u(\xi)}\hat f(0^+)\right)d\xi=\frac 1{2\pi}\hat f(0^+)\widehat{H_u(u)}$$ as $\eta$ tends to $0$. 
It completes the proof.
  \end{proof}
  As an immediate corollary, we get the following result
  \begin{corollary}\label{CommutatorXstar}
   Let $u\in H^{\frac 12}_+(\R)$. Let $f\in  {\rm Dom}X^*$. Then,  
  $$\left[X^*,H_u^2\right ]f=\frac i{2\pi}\left(I_+(f)H_u(u)-\langle f,u\rangle u\right)\; \text{ in }L^2_+(\R).$$
  \end{corollary}
  The following lemma can be compared to similar ones in \cite{GaGM26}, \cite{GaG26} and \cite{GH26}. 
  \begin{lemma}\label{EigenfunctionDomXStar}
  Let $u\in H^{\frac 12}_+(\R)$.The eigenfunctions associated to non zero eigenvalues of $H_u^2$ belong to the domain of $X^*$. Moreover, if $\varphi$ is an eigenfunction with $H_u^2(\varphi)=\sigma^2\varphi$, $\sigma\neq 0$,
  $$I_+(\varphi)=\frac 1{\sigma^2}\langle u,H_u(\varphi)\rangle$$ 
  and
  $$(\sigma^2-H_u^2)X^*(\varphi)= \frac{i}{2\pi}I_+(\varphi )H_uu-\frac i{2\pi}\langle \varphi,u\rangle u.$$
  \end{lemma}
  \begin{proof}
  We first prove that $\varphi$ belongs to the domain of $I_+$. We write, 
  $$\hat \varphi(\xi)=\frac 1{\sigma^2} \widehat{H_u^2(\varphi)}(\xi)=\frac 1{2\pi \sigma^2}\int_0^\infty \hat u(\xi+\eta)\overline{\widehat{H_u(\varphi)}(\eta)}d\eta.$$ As the right hand side has a pointwise limit as $\xi\to 0$, we get the result with
  $$I_+(\varphi)=\frac 1{\sigma^2}\langle u,H_u(\varphi)\rangle$$ as expected.\\
  We turn to the proof of the second part of the lemma. We have to prove that $$\Psi_\eta:=\frac{ S(\eta)^*\varphi-\varphi}{\eta}$$ has a limit as $\eta$ tends to $0^+$. We first choose $u_\varepsilon$ a rational function such that $\Vert u-u_\varepsilon\Vert_{H^{\frac 12}}<\varepsilon$ with $\varepsilon$ sufficiently small so that 
  $\Vert H_u^2-H_{u_\varepsilon}^2\Vert_{\mathcal L(L^2)}<\sigma^2$. We consider $$T_\varepsilon :=\mathrm{Id}-\frac{(H_u^2-H_{u_\varepsilon}^2)}{\sigma^2}$$ which is an invertible operator with this choice of $\varepsilon$. One has 
  $$T_\varepsilon(\varphi)=\varphi-\frac{(H_u^2-H_{u_\varepsilon}^2)}{\sigma^2}\varphi=\frac 1{\sigma^2}H_{u_\varepsilon}^2\varphi.$$ 
  Hence 
  $$T_\varepsilon\Psi_\eta=\frac {S(\eta)^*- \mathrm{Id}}\eta\frac 1{\sigma^2}H_{u_\varepsilon}^2\varphi-\left[\frac{S(\eta)^*}\eta, T_\varepsilon\right]\varphi.$$
  As $H_{u_\varepsilon}^2(\varphi)$ is a rational function in $L^2_+(\R)$, it belongs to the domain of $X^*$ and $$\frac {S(\eta)^*- \mathrm{Id}}\eta\frac 1{\sigma^2}H_{u_\varepsilon}^2\varphi$$ has a limit as $\eta\to 0$. It remains to consider the term $\left[\frac{S(\eta)^*}\eta, T_\varepsilon\right]\varphi$.
  As $$\left[\frac{S(\eta)^*}\eta, T_\varepsilon\right]\varphi=\frac 1{\sigma^2}\left[\frac{S(\eta)^*}\eta, H_{u_\varepsilon}^2\right]\varphi-\frac 1{\sigma^2}\left[\frac{S(\eta)^*}\eta, H_u^2\right]\varphi$$ we use Lemma \ref{Commutator} to obtain the existence of the limit as $\eta\to 0$.\\
To identify $(\sigma^2-H_u^2)X^*\varphi$, we compute
$$(\sigma^2-H_u^2)\Psi_\eta=\left[\frac{S(\eta)^*}\eta,H_u^2\right]\varphi.$$
Another use of Lemma \ref{Commutator} gives 
$$(\sigma^2-H_u^2)\Psi_\eta\td_\eta,0\frac 1{2\pi}\left(I_+(\varphi)H_u(u)-\langle \varphi,u\rangle u\right)$$ which equals $(\sigma^2-H_u^2)(-iX^*)(\varphi)$ and completes the proof.

  \end{proof}
As a corollary, we recover a result established in \cite{GPJEMS} in the special case of rational functions.
\begin{corollary}\label{notperp}
Let $u\in H^{\frac 12}_+(\R)$. For any non--zero eigenvalue $\sigma^2$ of $H_u^2$, $u$ is not orthogonal to $\ker(H_u^2-\sigma^2)$.
\end{corollary}
\begin{proof}
Let $\sigma$ be such that $u\perp \ker(H_u^2-\sigma^2)$. We prove that $\ker(H_u^2-\sigma^2)=\{0\}$. We first prove that $I_+(\varphi)=0$ for any $\varphi\in \ker(H_u^2-\sigma^2)$. 
By  Lemma \ref{EigenfunctionDomXStar}, 
$$I_+(\varphi)=\frac 1{\sigma^2} \langle u,H_u(\varphi)\rangle=0,$$ {since $H_u(\varphi)\in \ker(H_u^2-\sigma^2)$}.
On the other hand, by  Lemma \ref{EigenfunctionDomXStar}, 
$$(\sigma^2-H_u^2)X^*\varphi= \frac i{2\pi}\left(I_+(\varphi)H_u(u)-\langle \varphi,u\rangle u\right)=0.$$ Hence $X^*\varphi\in \ker(H_u^2-\sigma^2)$ so that $X^*$ acts on $\ker(H_u^2-\sigma^2)$. It implies the same for $S(\eta)^*$ so that $S(\eta)^*\varphi\in \ker(H_u^2-\sigma^2)$. Hence, $\hat\varphi(\eta)=I_+(S(\eta)^*\varphi)=0$ for any $\eta\ge 0$. We conclude that $\varphi=0$ so that $\ker(H_u^2-\sigma^2)=\{0\}$.
\end{proof}

\section{Proof of Theorem \ref{maintheorem}}
In this section, we assume that $u_0$ satisfies the assumption of Theorem \ref{maintheorem}.
In view of Corollary \ref{notperp}, we can also assume that the eigenfunctions $\varphi_j$ are chosen so that $\langle u_0,\varphi_j\rangle$ are positive real numbers. We denote by $$\mathcal M:=\left\{\langle u_0,\varphi_j\rangle,\, j\ge 1\right\}\subset (0,\infty ).$$
For $\mu\in \mathcal M$, we introduce $$J_\mu:=\{j\ge 1;\; \langle u_0,\varphi_j\rangle=\mu\}.$$
Observe that, as $(\langle u_0,\varphi_j\rangle)_j\in\ell^2(\Z_+^*)$, the set $\mathcal M$ is countable --- or finite if $u_0$ is a rational function ---  but, for each $\mu\in\mathcal M$, the set $J_\mu$ is a finite set of integers.  We also consider the orthogonal projector $\mathcal P_\mu$ on $E_\mu:=\mathrm{Span}\{\varphi_j\}_{j\in J_\mu}$ so that $$u_0=\sum_{\mu\in\mathcal M}\mathcal P_\mu(u_0),\quad \mathcal P_\mu(u_0)=\mu \sum_{j\in J_\mu}\varphi_j. $$
The following theorem is a more precise reformulation of Theorem \ref{maintheorem}. {Notice that we have changed slightly the notation, replacing the set $\mathcal V$ of velocities by the set $\mathcal M$. The connection is $c=\mu ^2/2\pi$.}
\begin{theorem}\label{Main2}
Assume $u_0$ in $H^{\frac 12}_+(\R)$ is such that all positive eigenvalues of $H_{u_0}^2$ are simple. With the notation above, for every $\mu \in \mathcal M$, the formula
$$\forall z \in \C_+,  \quad R_\mu ^0(z)=\frac{1}{2i\pi}I_+((\mathcal P_\mu X^*-z)^{-1}\mathcal P_\mu u_0)$$
defines a rational function $R_\mu ^0\in L^2_+(\R)$, with denominator of degree $N_\mu=|J_\mu|$. Denote by $u$ the Szeg\H{o} solution with initial datum  $u_0$, and by  $R_\mu $ the
Szeg\H{o} solution with initial datum $R_\mu ^0$. Then
 $$\lim_{t\to \pm \infty} \left\Vert u(t,\cdot)-\sum_{\mu\ge \e} R_\mu (t,\cdot )\right\Vert_{H^{\frac 12}}\td_\e,{0^+} 0 .$$
 Furthermore,  there exists a smooth function $G_\mu $ on $\T ^{J_\mu}\times \R$, rational in $x$ with denominator of degree $N_\mu$, such that 
 $$R_\mu \left (t,x+t\frac{\mu^2}{2\pi}\right )=G_\mu (t\omega_\mu ,x), $$
 where $\omega _\mu :=(\sigma_j^2)_{j\in J_\mu}$ is the family of eigenvalues of $H_{u_0}^2$ on $E_\mu$.
\end{theorem}
\begin{remark} Notice that $\Vert u_0\Vert_{L^2}^2=\sum_{\mu\in\mathcal M}\mu^2N_\mu <\infty$. In particular, the sum over $\mu\ge {\e}$ is a finite sum. \end{remark}
\begin{remark}
In the special case $|J_\mu|=1$, $R_\mu $ is a soliton solution. If $|J_\mu|>1$, $R_\mu$ can be seen as a traveling quasi--periodic breather solution in the sense of Definition \ref{breather}.
\end{remark}
\subsection{A weak convergence result}
As a first step, we are going to identify the behaviour of $u(t,\cdot+\frac t{2\pi}\mu^2)$ in $H^{1/2}_+(\R)$ as $t\to \infty$ for every $\mu \in \mathcal M$. 
Recall that 
\begin{eqnarray*}
u(t,z+\frac t{2\pi}\mu^2)&=&\Omega_t({\rm e}^{-itH_{u_0}^2}u_0)(z+\frac t{2\pi}\mu^2)\\
&=& \sum_\nu\nu\sum_{j\in J_\nu}\Omega_t({\rm e}^{-it\sigma_j^2}\varphi_j)(z+\frac t{2\pi}\mu^2).
\end{eqnarray*}

Let us compute $L_{u_0}(t)$. Writing as before $${\rm e}^{-itH_{u_0}^2}u_0=\sum_\nu \nu\sum_{j\in J_\nu}{\rm e}^{-it\sigma_j^2} \varphi_j,$$ we get
\begin{eqnarray*}
L_{u_0}(t)f=\frac 1{2\pi}\int_0^t\sum_{\nu,\rho\in \mathcal M} \sum_{j\in J_\nu}\sum_{k\in J_{\rho}}\nu\,\rho\,{\rm e}^{is(\sigma_k^2-\sigma_j^2)}\langle f,\varphi_k\rangle \varphi_j ds
\end{eqnarray*}
which gives
\begin{equation}\label{Lu2}
L_{u_0}(t)f=\frac t{2\pi}\sum_\nu \nu^2\mathcal P_\nu(f) +\sum_{\nu,\rho\in\mathcal M} \nu\rho \sum_{\substack{j\in J_\nu,\\k\in J_{\rho} \\ j\neq k}}a_{j,k}(t)\langle f,\varphi_k\rangle \varphi_j\end{equation}
with,  for  $j\neq k$,
\begin{equation}\label{eq:ajk}
a_{j,k}(t)=\frac{{\rm e}^{it(\sigma_k^2-\sigma_j^2)}-1}{\pi(\sigma_k^2-\sigma_j^2)}.
\end{equation}

At this stage we observe the following fact.
\begin{remark}\label{rem:invertible}
As for any $h\in E_\mu$, for any $z\in \C_+$, 
\begin{eqnarray*}
\mathrm{Im}\left\langle \mathcal P_{\mu}\left(X^*-z+L_{u_0}(t)-\mu^2\frac t{2\pi}\right)h,h\right\rangle &=&\mathrm{Im}\langle X^*h,h\rangle-\mathrm{Im}(z) \Vert h\Vert_{L^2}^2\\
&\le& -\mathrm{Im}(z) \Vert h\Vert_{L^2}^2,
\end{eqnarray*}
the operator $$\mathcal P_{\mu}\left(X^*-z+L_{u_0}(t)-\mu^2\frac t{2\pi}\right):E_\mu\to E_\mu$$ is injective. Since $E_\mu$ is finite dimensional, this operator is bijective and $$\left(\mathcal P_{\mu}\left(X^*-z+L_{u_0}(t)-\mu^2\frac t{2\pi}\right)\right)^{-1}$$ is well defined as a linear operator on $E_\mu$. Furthermore, it is uniformly bounded as $t\to \infty$, since
$$\left \Vert \left(\mathcal P_{\mu}\left(X^*-z+L_{u_0}(t)-\mu^2\frac t{2\pi}\right)\right)^{-1}f\right \Vert \leq \frac{1}{\Im z}\Vert f\Vert, \quad f\in E_\mu .$$
\end{remark}

We establish a crucial lemma relating the above spectral theory of $H_{u_0}^2$ with the operator $\Omega_t$ from Corollary \ref{cor:explicit} arising in  the explicit formula \eqref{Explicit}.

\begin{lemma}\label{tecnic2}
Assume $u_0\in H^{\frac 12}_+$. \\
For any $\nu\neq \mu$, $\mu,\nu\in \mathcal M$, any $f_\nu\in E_\nu$,
$$\Omega_t\left(f _\nu\right)\left(z+\frac t{2\pi}\mu^2\right)\td_t,\infty 0 \quad\mathrm{ a.e } \; z\in\C_+.$$
More generally, if $g\in E_\mu^\perp\cap\overline{ \mathrm{Im}H_{u_0}}$,
$$\Omega_t\left(g\right)\left(z+\frac t{2\pi}\mu^2\right)\td_t,\infty 0 \quad\mathrm{ a.e } \; z\in\C_+,$$
with uniform convergence if $g$ belongs to a compact subset of $E_\mu^\perp\cap\overline{ \mathrm{Im}H_{u_0}}$.
For $f_\mu\in E_\mu$, one has
$$ \Omega_t\left(f_\mu\right)\left(z+\frac t{2\pi}\mu^2\right)-\frac 1{2i\pi}I_+\left(\Psi_\mu(t,z)\right)\td_t,\infty 0  \quad\mathrm{ a.e } \; z\in\C_+$$
where \begin{equation}\label{Psi}\Psi_{\mu}(t,z):=\left(\mathcal P_{\mu}\left(X^*-z+L_{u_0}(t)-\mu^2\frac t{2\pi}\right)\right)^{-1}f_\mu
\end{equation}

\end{lemma}
\begin{proof}
Let $\nu\neq \mu$, $\nu,\mu\in \mathcal M$.
We recall that
$$\left(L_{u_0}(t)-\mu^2\frac t{2\pi}\right)f_\nu=(\nu^2-\mu^2)\frac t{2\pi}f_\nu+g_\nu(t)$$
where $g_\nu(t)$ is defined by 
$$g_\nu(t):=\sum_{\rho\in\mathcal M} \rho\nu \sum_{\substack{j\in J_\rho,\;k\in J_{\nu}\\ j\neq k}}a_{j,k}(t)\langle f_\nu,\varphi_k\rangle \varphi_j, $$
and $a_{j,k}(t)$ is given by \eqref{eq:ajk}. 
Notice that the family $\{ g_\nu (t)\}_{t\in \R}$ is uniformly bounded in $L^2_+$.
We define 
$$\Psi_{\nu, \mu}(t)=\frac{ 2\pi f_\nu}{( \nu^2-\mu^2)}$$ so that
$$\left(L_{u_0}(t)-\mu^2\frac t{2\pi}\right)\frac{\Psi_{\nu,\mu} }t=f_\nu+\frac{ 2\pi g_\nu(t)}{t( \nu^2-\mu^2)}.$$
We compute
\begin{eqnarray*}
&&\left [(X^*-z)+(L_{u_0}(t)-\frac{t}{2\pi}\mu^2)\right]^{-1}f_\nu\\
&=& \left [(X^*-z)+(L_{u_0}(t)-\frac{t}{2\pi}\mu^2)\right]^{-1} \left[(L_{u_0}(t)-\frac{t}{2\pi}\mu^2)\frac{\Psi_{\nu, \mu}}t-\frac{ 2\pi g_\nu(t)}{t( \nu^2-\mu^2)}\right]\\
&=&\frac{\Psi_{\nu, \mu}}t-\left[(X^*-z)+(L_{u_0}(t)-\frac{t}{2\pi}\mu^2)\right]^{-1}\left[(X^*-z)\frac{\Psi_{\nu, \mu}}t+\frac{ 2\pi g_\nu(t)}{t( \nu^2-\mu^2)}\right].
 \end{eqnarray*}
 Computing $\Omega_t(f_\nu)(z+\frac{t}{2\pi}\mu^2)$, we get
$$
\Omega_t(f_\nu)(z+\frac{t}{2\pi}\mu^2)=\frac 1t I_+(\Psi_{\nu, \mu})-\Omega_t\left((X^*-z)\frac{\Psi_{\nu, \mu}}t+\frac{ 2\pi g_\nu(t)}{t( \nu^2-\mu^2)}\right)(z+\frac{t}{2\pi}\mu^2).$$
 
Using the pointwise estimate \eqref{Pointwise} of $\Omega_tf$, the fact that  $\Psi_{\nu,\mu}, g_\nu(t) \in L^2_+$ with $\Vert g_\nu(t)\Vert_{L^2}\le C$ and $\Psi_{\nu,\mu}\in \mathrm{ Dom} X^* $ give the first part of Lemma \ref{tecnic2}.

We turn to the second part and consider $\Omega_t(g)(z+\frac{t}{2\pi}\mu^2)$ for $g\in E_\mu^\perp\cap\overline{ \mathrm{Im}H_{u_0}}$. We obtain the result  by a density argument based on the estimate \eqref{Pointwise} since, by linearity, the result holds if $g$ is a finite sum of elements in $E_\nu$ for $\nu\neq \mu$ from the first part. The claimed uniform convergence on compact subsets of $E_\mu^\perp\cap\overline{ \mathrm{Im}H_{u_0}}$ follows as well.

 Finally, we  write for $\Psi_\mu(t,z)\in E_\mu$ defined in \eqref{Psi}
\begin{eqnarray*}
&&\left(X^*-z+L_{u_0}(t)-\mu^2\frac t{2\pi}\right)\Psi_\mu(t,z)\\
&=&\mathcal P_\mu\left(X^*-z+L_{u_0}(t)-\mu^2\frac t{2\pi}\right)\Psi_\mu(t,z)\\
&&+(I-\mathcal P_\mu)\left(X^*-z+L_{u_0}(t)-\mu^2\frac t{2\pi}\right)\Psi_\mu(t,z)\\
&=&f_\mu+(I-\mathcal P_{\mu})(X^*+L_{u_0}(t)-\frac{\mu ^2t}{2\pi})\Psi_{\mu}(t,z).
\end{eqnarray*}
Hence, \begin{align*}
\Omega_t(f_\mu)\left(z+t\frac{\mu^2}{2\pi}\right)&=\frac 1{2i\pi}I_+(\Psi_\mu(t,z))-\\
&\Omega_t\left((I-\mathcal P_{\mu})(X^*+L_{u_0}(t)-\frac{\mu ^2t}{2\pi})\Psi_{\mu}(t,z)\right)\left(z+t\frac{\mu^2}{2\pi}\right).
\end{align*}
At this stage, we observe  that, given $z\in \C_\mu, f_\mu \in E_\mu$, in view of Remark \ref{rem:invertible}, the family $\{ \Psi_\mu (t,z)\}_{t\in \R}$ is bounded in $E_\mu $, hence it is relatively compact, since $E_\mu $ is finite dimensional.
\\
Consequently, using \eqref{Lu2}, $(I-\mathcal P_{\mu})(X^*+L_{u_0}(t)-\frac{\mu ^2t}{2\pi})\Psi_{\mu}(t,z)$ belongs to a compact subset of $E_\mu^\perp\cap\overline{ \mathrm{Im}H_{u_0}}$, therefore from the second part the last term in the above identity tends to zero (all the operators involved here act on $\overline{ \mathrm{Im}H_{u_0}}$).  This completes the proof of  Lemma \ref{tecnic2}.
\end{proof}
Using this result, we infer the following 
\begin{corollary}\label{pointwiselimit}
The following holds
$$
\forall z\in\C_+,\quad u\left (t,z+\frac t{2\pi}\mu^2\right )-\frac 1{2i\pi}I_+\left(\psi_\mu(t,z)\right)\td_t,\infty 0,$$ 
where
\begin{equation}\label{psi}\psi_\mu(t,z):=\left(\mathcal P_{\mu}\left(X^*-z+L_{u_0}(t)-\mu^2\frac t{2\pi}\right)\right)^{-1}\mathcal P_\mu({\rm e}^{-itH_{u_0}^2}u_0) .\end{equation}
\end{corollary}

\begin{proof}
The result is a direct consequence of Lemma \ref{tecnic2}. Just write 
$$
u\left (t,z+\frac t{2\pi}\mu^2\right )
=\Omega_t({\mathrm e}^{-itH_{u_0}^2} u_0)\left (z+\frac t{2\pi}\mu^2\right )$$

and use $${\mathrm e}^{-itH_{u_0}^2} u_0=\mathcal P_\mu({\rm e}^{-itH_{u_0}^2}u_0)+f_{\mu^\perp}$$ with $f_{\mu^\perp}(t):=(I-\mathcal P_\mu)({\rm e}^{-itH_{u_0}^2}u_0)$ staying in a compact subset of $ E_\mu^\perp\cap\overline{ \mathrm{Im}H_{u_0}}$.

 \end{proof}
\subsection{The traveling quasi--periodic breather}
Lemma \ref{pointwiselimit} can be reformulated as 
\begin{equation}\label{ptw}
\forall z\in\C_+,\quad u\left (t,z+\frac t{2\pi}\mu^2\right )-R_\mu \left (t,z+\frac t{2\pi}\mu^2\right )\td_t,\infty 0,
\end{equation}
where
\begin{equation}\label{R_mu}
R_\mu (t,z):=\frac{1}{2i\pi}I_+\left(\mathcal P_{\mu}(X^*-z+L_{u_0}(t))^{-1}\mathcal P_\mu({\mathrm e}^{-itH_{u_0}^2}u_0)\right ).
\end{equation}
We are going to identify the function $R_\mu$ as some solution to the cubic Szeg\H{o} equation. 
Firstly, we observe that the function $\psi_\mu$ from Lemma \ref{pointwiselimit} is defined in \eqref{psi} through the inversion of a matrix $A_\mu (t,z)$ of size $|J_\mu|$, given by
$$\left(A_\mu(t,z)\right)_{j,k}=\left\{\begin{array}{ccl}
&\langle X^*\varphi_j,\varphi_j\rangle -z&\text{ if } j=k,\\
&\langle X^*\varphi_k,\varphi_j\rangle +\mu^2 a_{j,k}(t)&\text{ otherwise. }
\end{array}\right.$$
We denote by  $X_\mu$  the column of components of   $\mathcal P_\mu({\rm e}^{-itH_{u_0}^2}u_0)$ in the basis $(\varphi_j)_{j\in J_\mu}$, given, for $1\le j\le |J_\mu|$, by
$(X_\mu)_j:=\mu\mathrm{e}^{-it\sigma_j^2}$.\\
On the other hand, we can write $H_{u_0}(\varphi_k)=\sigma_k \mathrm{e}^{i\alpha_k}\varphi_k$ for some $\alpha_k\in \T$, so that, for $k\in J_\mu$,
$$I_+(\varphi_k)=\frac 1{\sigma_k^2}\langle \varphi_k,H_{u_0}(u_0)\rangle=\frac{ \mu\mathrm{e}^{-i\alpha_k}}{\sigma_k}$$
which  eventually leads to
\begin{equation}\label{Rmu2}
R_\mu\left (t,z+\frac t{2\pi}\mu^2\right )=\frac 1{2i\pi}\langle \left(A_\mu(t,z)\right)^{-1}X_\mu ,Y_\mu \rangle
\end{equation}
where $Y_\mu:=\left(\frac{ \mu\mathrm{e}^{-i\alpha_k}}{\sigma_k}\right)_{k\in J_\mu}$.

We are going to reformulate equation \eqref{Rmu2} by introducing variables on the torus $\T^{J_\mu}$. 
Given $\theta \in \T^{J_\mu}$, we define $L_\mu (\theta )=E_\mu \to E_\mu $ and $\Omega_\theta:E_\mu \to E_\mu $ by 
$$\lan L_\mu (\theta) \varphi_k,\varphi_j\ran :=(1-\delta_{j,k})\frac{\mu ^2(\mathrm{e}^{i(\theta _k-\theta _j)}-1)}{2i\pi (\sigma_k^2-\sigma_j^2)},\quad \Omega_\theta \varphi_j=\mathrm{e}^{i\theta_j}\varphi_j,\quad j,k\in J_\mu .$$
\begin{proposition}\label{invertibility}
The operator $\mathcal P_\mu X^* +L_\mu (\theta ):E_\mu \to E_\mu $ has no real eigenvalue.
\end{proposition}
A crucial step in the proof is the following commutator lemma.
\begin{lemma}\label{lem:com}
For every $f\in E_\mu$,
$$[\mathcal P_\mu X^*+L_\mu (\theta),H_{u_0}\Omega _\theta]f=\frac{1}{2i\pi}\overline{I_+(f)}\Omega_\theta ^*\mathcal P_\mu u_0.$$
\end{lemma}
\begin{proof}
We recall the following identities for $j,k\in J_\mu, j\ne k$.
\begin{eqnarray}
&&\Im \lan X^*\varphi_j,\varphi_j\ran =-\frac{\mu ^2}{4\pi \sigma_j^2}, \label{eq:X*diag}\\
 &&\lan X^*\varphi_k,\varphi_j\ran =\frac{i\mu^2}{2\pi (\sigma_k^2-\sigma_j^2)}\left (\frac{\sigma_j}{\sigma_k}\mathrm{e}^{i(\alpha_j-\alpha_k)}-1  \right ), \label{eq:X*offdiag}
\end{eqnarray}
We first calculate, recalling that $H_{u_0}$ is antilinear,
$$
\lan [\mathcal P_\mu X^*, H_{u_0}\Omega _\theta ]\varphi_k,\varphi_j\ran =\sigma_k \mathrm{e}^{i(\alpha_k-\theta_k)}\lan X^*\varphi_k,\varphi_j\ran -\sigma_j\mathrm{e}^{i(\alpha_j-\theta_j)}\lan \varphi_j,X^*\varphi_k\ran $$
If $j=k$, this leads, in view of \eqref{eq:X*diag}, 
\begin{equation}\label{eq:comdiag}
\lan [\mathcal P_\mu X^*, H_{u_0}\Omega _\theta ]\varphi_j,\varphi_j\ran =\frac{\mu ^2}{2i\pi \sigma_j}\mathrm{e}^{i(\alpha_j-\theta_j)}.
\end{equation}
If $j\ne k$, this leads, in view of \eqref{eq:X*offdiag},
$$
\lan [\mathcal P_\mu X^*, H_{u_0}\Omega _\theta ]\varphi_k,\varphi_j\ran =\frac{i\mu ^2}{2\pi(\sigma_k^2-\sigma_j^2)}\left ( \sigma_j\mathrm{e}^{i(\alpha_j-\theta_k)}-\sigma_k \mathrm{e}^{i(\alpha_k-\theta_k)} +\frac{\sigma_j^2}{\sigma_k}\mathrm{e}^{i(\alpha_k-\theta_j)}-\sigma_j \mathrm{e}^{i(\alpha_j-\theta_j)}
\right )
$$
On the other hand, 
$$\lan [L_\mu (\theta), H_{u_0}\Omega _\theta ]\varphi_k,\varphi_j\ran =\sigma_k \mathrm{e}^{i(\alpha_k-\theta_k)}\lan L_\mu (\theta)\varphi_k,\varphi_j\ran -\sigma_j\mathrm{e}^{i(\alpha_j-\theta_j)}\lan \varphi_j,L_\mu (\theta)\varphi_k\ran ,$$
which, for $j\ne k$,  leads to
$$\lan [L_\mu (\theta), H_{u_0}\Omega _\theta ]\varphi_k,\varphi_j\ran =\frac{\mu ^2}{2i\pi (\sigma_k^2-\sigma_j^2)}\left (\sigma_k\mathrm{e}^{i(\alpha_k-\theta_j)}-\sigma_k \mathrm{e}^{i(\alpha_k-\theta_k)}+ \sigma_j\mathrm{e}^{i(\alpha_j-\theta_k)} -\sigma_j \mathrm{e}^{i(\alpha_j-\theta_j)}   \right )    $$
Summing the two above contributions, we infer, for $j\ne k$,
\begin{equation}\label{eq:comoffdiag}
\lan [\mathcal P_\mu X^*+L_\mu (\theta), H_{u_0}\Omega _\theta ]\varphi_k,\varphi_j\ran =\frac{\mu^2}{2i\pi\sigma_k}\mathrm{e}^{i(\alpha_k-\theta_j)}.
\end{equation}
Finally, comparing \eqref{eq:comdiag} and \eqref{eq:comoffdiag}, we have, for all $j,k\in J_\mu$,
$$\lan [\mathcal P_\mu X^*+L_\mu (\theta), H_{u_0}\Omega _\theta ]\varphi_k,\varphi_j\ran =\frac{\mu^2}{2i\pi\sigma_k}\mathrm{e}^{i(\alpha_k-\theta_j)},$$
which precisely means that
\begin{align*}
[\mathcal P_\mu X^*+L_\mu (\theta), H_{u_0}\Omega _\theta ]f&=\frac{\mu^2}{2i\pi}\sum_{k,j\in J_\mu}\lan \varphi_k,f\ran \frac{\mathrm{e}^{i\alpha_k}}{\sigma_k}\mathrm{e}^{-i\theta_j}\varphi_j\\
&=\frac{1}{2i\pi}\overline{I_+(f)}\Omega_\theta^*\mathcal P_\mu u_0.
\end{align*}
\end{proof}
Let us now complete the proof of Proposition \ref{invertibility}. For any $\lambda \in \R$, set $$V_\mu (\theta,\lambda ):=\ker(\mathcal P_\mu X^*+L_\mu (\theta)-\lambda).$$  Given $f\in V_\mu (\theta, \lambda)$, taking the imaginary part of the inner product with $f$ of both sides of 
$$(\mathcal P_\mu X^*+L_\mu (\theta)-\lambda)f=0$$
we get $I_+(f)=0$. Using Lemma \ref{lem:com}, we infer 
$$(\mathcal P_\mu X^*+L_\mu (\theta)-\lambda)H_{u_0}\Omega_\theta f=\frac{1}{2i\pi}\overline{I_+(f)}\Omega_\theta ^*\mathcal P_\mu u_0=0.$$
This means that $V_\mu (\theta, \lambda)$ is preserved by $H_{u_0}\Omega_\theta$, hence by $$(H_{u_0}\Omega_\theta)^2=H_{u_0}^2.$$
If $V_\mu (\theta, \lambda)\ne \{ 0\}$, it therefore contains an eigenvector $\varphi_j$ of $H_{u_0}^2$, which is a contradiction, since $$I_+(\varphi_j)=\frac{\mu}{\sigma_j}\mathrm{e}^{-i\alpha_j}\ne 0.$$
The proof of Proposition \ref{invertibility} is complete.

In view of Proposition \ref{invertibility}, we can introduce the following  function $G_\mu :\T^{J_\mu}\times \overline{\C_+}$,
\begin{equation}\label{eq:Gmu}
G_\mu (\theta,z):=\frac{1}{2i\pi}I_+((\mathcal P_\mu X^*+L_\mu (\theta)-z)^{-1}\Omega_\theta^*\mathcal P_\mu ( u_0)),
\end{equation}
so that $R_\mu (t,z)$ extends to real values of $z$ and, $\forall t\in \R, \forall x\in \R,$ 
\begin{equation}\label{RG}
 R_\mu \left (t,x+t\frac{\mu^2}{2\pi}\right )=G_\mu  (t\omega_\mu, x),\quad \omega_\mu :=(\sigma_j^2)_{j\in J_\mu}.
\end{equation}

Notice that $G_\mu (\theta,\cdot )$ is also a rational function in $L^2_+$, with a denominator of degree $|J_\mu|$. Furthermore, notice that
\begin{align}
& [i\partial_tR_\mu -\Pi (|R_\mu |^2R_\mu)] \left (t, x+t\frac{\mu^2}{2\pi}\right )=H_\mu (t\omega_\mu, x), \label{RmuHmu}\\
&H_\mu (\theta ,x):=\left (i\sum_{j\in J_\mu}\sigma_j^2\partial_{\theta_j}G_\mu -i\frac{\mu^2}{2\pi}\partial_xG_\mu -\Pi (|G_\mu |^2G_\mu)\right )(\theta,x). \label{eq:Hmu}
\end{align}

Furthermore, the pointwise convergence \eqref{ptw} can be rephrased as 
\begin{equation}\label{weakstar}
u\left (t,\cdot +t\frac{\mu^2}{2\pi}\right )-R_\mu \left (t,\cdot +t\frac{\mu^2}{2\pi}\right )\rightharpoonup 0
\end{equation}
as $t\to \infty$, for the weak $*$ convergence in $H^{1/2}_+(\R)$. 

In order to check that $R_\mu $ is a solution of the cubic Szeg\H{o} equation, we shall appeal to an auxiliary result on quasi-periodic functions, of which we first recall the definition.
\begin{definition}
A function $f:\R \to \C$ is said quasi--periodic if there exist $N\in \N,$ $ \omega \in \R^N$ and a continuous function $F:\T^N\to \C$  such that 
\[ \forall t\in \R,\quad f(t)=F(t\omega).\]
\end{definition}
\begin{proposition}\label{prop:cancel}
Let $f$ be a quasi--periodic function such that there exist a $C^1$ function $\e $ and a continuous function $\delta $ such that
$$f(t)=\e '(t)+\delta (t),\quad \e(t)\td_t,\infty0,\quad \delta (t)\td_t,\infty 0.$$
 Then $f(t)=0$ for every $t\in \R$.
\end{proposition}
\begin{proof} It relies on the following elementary lemma.
\begin{lemma}\label{independence}
If $f$ is quasiperiodic, there exist $ (d,D)\in \N \times \N,$ a continuous function $G:\T^d \to \C, $ and  $\alpha \in \R^d$ with $\Q$--linearly independent components, such that
\[ \forall t\in \R,\quad f(Dt)=G(t\alpha ).\]
\end{lemma}
\begin{proof}Indeed, let $\Omega :={\mathrm{Vect}}_\Q\{\omega_1,\dots, \omega_N\}$ and $d:=\dim _\Q \Omega$. Let $(\alpha_1,\dots ,\alpha_N)$ be a basis of the $\Q$--linear space $\Omega$. There exists $D\in \N$, and $M:=(m_{jk})_{1\leq j\leq N, 1\leq k\leq d}\in \Z^{N\times d}$ such that
\[\omega =\frac{1}{D}M\alpha .\]
Then, with $G(\theta ):=F(M\theta )$ for every $\theta \in \T^d$, we have \[f(Dt)= F(Dt\omega )=G(t\alpha ).\]
\end{proof}
Let us complete the proof of Proposition \ref{prop:cancel}. Using Lemma \ref{independence}, we have 
$$G(t\alpha )=f(Dt)=\tilde \e '(t)+\tilde \delta (t),$$
where 
$$\tilde \e(t):=D^{-1}\e(Dt),\quad \tilde \delta (t):=\delta (Dt)$$
are going to $0$ as $t\to \infty$. 
 Let $k\in \Z^d$. Since the components of 
$\alpha $ are $\Q$--linearly independent, we have, for every continuous function $H$ on $\T^d$,
\[ \frac 1{2T} \int_{-T}^T H(t\alpha )\, dt \td_T,\infty \hat H(0).\]
Indeed, the left hand side is a uniformly bounded family of linear forms on $C(\T^d)$ and the claimed convergence holds in the special cases \[H(\theta )=\mathrm{e}^{ik.\theta}, k\in \Z^d.\]
Applying this property to $H(\theta)=G(\theta )\mathrm{e}^{-ik.\theta}$, we conclude that
\begin{align*}
\hat G(k)&=\lim_{T\to \infty}\int_{-T}^T [\tilde \e '(t)+\tilde \delta (t]\mathrm{e}^{-itk.\alpha}\, dt\\
&=\lim_{T\to \infty}\left (\frac{\e(T)\mathrm{e}^{-iTk.\alpha}-\e(-T)\mathrm{e}^{iTk.\alpha}}{2T}+\frac{1}{2T}\int_{-T}^T [ik.\alpha\, \tilde \e (t)+\tilde \delta (t)]\mathrm{e}^{-itk.\alpha}\, dt\right )\\
&=0,
\end{align*}
by the Cesaro mean theorem. Hence $G=0$ and $f=0$.
\end{proof}
Given $\zeta \in \mathscr S(\R)\cap L^2_+(\R)$, we set
$$h_\mu(t):=\int_\R H_\mu (t\omega_\mu ,x)\overline \zeta (x)\, dx.$$
Obviously, $h_\mu $ is quasi--periodic. Furthermore, in view of \eqref{RmuHmu} and of the cubic Szeg\H{o} equation \eqref{Szego} for $u$, we have
$$h_\mu (t)=\e_\mu '(t)+\delta_\mu (t),$$
with
\begin{align*}
\e_\mu (t)&:=-i\int_\R \left [u\left (t,\cdot +t\frac{\mu^2}{2\pi}\right )-R_\mu \left (t,\cdot +t\frac{\mu^2}{2\pi}\right )\right ]\overline \zeta (x)\, dx,\\
\delta_\mu (t)&:=\frac{-i\mu ^2}{2\pi}\int_\R \left [u\left (t,\cdot +t\frac{\mu^2}{2\pi}\right )-R_\mu \left (t,\cdot +t\frac{\mu^2}{2\pi}\right )\right ]\overline \zeta '(x)\, dx\\
&+\int_\R \left [|u|^2u\left (t,\cdot +t\frac{\mu^2}{2\pi}\right )-|R_\mu|^2R_\mu \left (t,\cdot +t\frac{\mu^2}{2\pi}\right )\right ]\overline \zeta (x)\, dx.
\end{align*}
Moreover, in view of \eqref{weakstar}, $\e_\mu (t)\td_t,\infty 0$, and $\delta_\mu (t)\td_t,\infty 0$. Applying Proposition \ref{prop:cancel}, we infer that $h_\mu $ is identically zero. Since this holds for every test function $\zeta \in \mathscr S(\R)\cap L^2_+(\R)$, we conclude that $H_\mu (t\omega_\mu ,x)= 0$ for every $t, x,$ and, in view of \eqref{RmuHmu},
that $R_\mu $ is a solution to the cubic Szeg\H{o} equation.

It remains to prove the convergence in $H^{\frac 12}_+$. For this, we shall need the following additional properties of $R_\mu$.

\begin{proposition}
For $R_\mu^0(x)=R_\mu (0,x)$, we have the following identities.
\begin{align}
&\Vert R_\mu^0\Vert_{L^2}^2=\Vert \mathcal P_\mu u_0\Vert_{L^2}^2. \label{RmuL2}\\
&\frac {1}{2\pi}\langle DR_\mu^0, R_\mu^0\rangle =\sum_{j\in J_\mu}\sigma_j^2. \label{RmuH1/2}
\end{align}
\end{proposition}
\begin{proof}
First we prove \eqref{RmuL2}, relying on an argument from Lemma 6.7 in \cite{GPJEMS}. From \eqref{R_mu}, we have
\[ \forall z\in \C_+, R_\mu^0(z)=\frac 1{2i\pi}I_+(\psi_\mu ^0(z,\cdot )),\quad  \psi_\mu^0(z,\cdot ):=(\mathcal P_\mu X^*-z)^{-1}\mathcal P_\mu u_0,\]
and this identity extends to $z=x\in \R$ in view of Proposition \ref{invertibility} with $\theta =0$. We infer, for every $x\in \R$,
\begin{align*}
|R_\mu ^0(x)|^2&=\frac{1}{4\pi^2}|I_+(\psi_\mu ^0(x,\cdot ))|^2\\
&=\frac{1}{\pi }\mathrm{Im}\langle \psi_\mu^0(x,\cdot ),X^* \psi_\mu^0(x,\cdot )\rangle \\
&=\frac{1}{\pi}\mathrm{Im}\langle \psi_\mu^0(x,\cdot ),\mathcal P_\mu X^* \psi_\mu^0(x,\cdot )\rangle \\
&=\frac{1}{\pi}\mathrm{Im}\langle \psi_\mu^0(x,\cdot ),\mathcal P_\mu u_0+x \psi_\mu^0(x,\cdot )\rangle \\
&=\frac{1}{\pi}\mathrm{Im}\langle \psi_\mu^0(x,\cdot ),\mathcal P_\mu u_0\rangle.
\end{align*}
Consequently,
\begin{equation}\label{eq:RmuL2}
 \Vert R_\mu ^0\Vert_{L^2}^2=\frac{1}{\pi}\mathrm{Im}\lim_{R\to \infty}\int_{-R}^R \langle \mathcal (\mathcal P_\mu X^*-x)^{-1}\mathcal P_\mu u_0,\mathcal P_\mu u_0\rangle\, dx.
 \end{equation}
Recall from Remark \ref{rem:invertible} that $\mathcal P_\mu X^* -z $ is invertible on $E_\mu$ for every $z\in\overline  \C_+$. Hence, by the Cauchy integral formula, we have, in $\mathcal L(E_\mu)$, 
\[ \int_{-R}^R (\mathcal P_\mu X^*-x)^{-1}\, dx =-\int_0^\pi (\mathcal P_\mu X^*-R\mathrm{e}^{i\theta})^{-1}iR\mathrm{e}^{i\theta}\, d\theta \td_R,\infty i\pi \mathrm{Id_{E_\mu}},\]
which, in view of \eqref{eq:RmuL2}, leads to \eqref{RmuL2}.

Let us come to the proof of \eqref{RmuH1/2}. Since a similar argument to the previous one seems more intricate for this identity, we are going to prove that, for every $j\in J_\mu$, $\sigma_j^2$ is an eigenvalue of $H_{R_\mu ^0}^2$. Since the degree of the denominator of $R_\mu ^0$ is $|J_\mu|$, this will immediately prove that
\[ \mathrm{Tr}(H_{R_\mu ^0}^2)=\sum_{j\in J_\mu} \sigma_j^2,\]
which, in view of \eqref{TrH2}, is \eqref{RmuH1/2}. Hence the proof reduces to the following lemma.
\begin{lemma}\label{spectralRmu}
For every $j\in J_\mu$, $z\in \overline {\C_+}$, define
\[ q_j(z):=\frac{1}{2i\pi} I_+((\mathcal P_\mu X^*-z)^{-1}\varphi_j).\]
Then $\Vert q_j\Vert_{L^2}=1$ and $H_{R_\mu ^0}q_j =\sigma_j\mathrm{e}^{i\alpha_j}q_j.$
\end{lemma}
Let us prove Lemma \ref{spectralRmu}. The first identity is a direct consequence of the above calculation, replacing $\mathcal P_\mu u_0$ by $\varphi_j$. For the second identity, we are going to rely on Proposition \ref{prop:cancel},  \eqref{weakstar} and Corollary \ref{cor:explicit} from the explicit formula. The approach has some similarity with  the proof of the conservation of energy in \cite{GL26}. 
Let us first assume $u_0\in H^s_+$ for some $s>1/2$. From \eqref{UHu}, we infer
\[ H_{u(t)}U(t)\varphi_j=\sigma_j \mathrm{e}^{i\alpha_j}\varphi_j,\]
so that, using  \eqref{UOmega},
\begin{equation}\label{HuOmegaphi}
H_{u(t)} \Omega_t \varphi_j=\sigma _j\mathrm{e}^{i\alpha_j-it\sigma_j^2}\Omega_t \varphi_j. 
\end{equation}
Identity \eqref{HuOmegaphi} still holds if $u_0\in H^{1/2}_+$, by a simple approximation argument and the strong continuity of the Szeg\H{o} flow on $H^{1/2}_+(\R)$. Let us now translate the $x$ variable by $t\mu ^2/2\pi$. Then we observe that 
\[ \Omega_t \varphi_j\left (x+t\frac{\mu^2}{2\pi}\right )=Q_j(t\omega_\mu, x),\]
where $Q_j:\T^{J_\mu}\times \R$ is defined by
\[Q_j(\theta,x):=\frac{1}{2i\pi}I_+((\mathcal P_\mu X^*+L_\mu (\theta)-x)^{-1}\varphi_j).\]
Notice that $Q_j(\theta,\cdot )$ is rational function in $L^2_+(\R)$, with a denominator of degree at most $|J_\mu|$, and with periodic coefficients. Consequently, the family $\left (\Omega_t \varphi_j\left (x+t\frac{\mu^2}{2\pi}\right )\right )_{t\in \R}$ belongs to a compact subset of $L^2_+(\R)$. In view of \eqref{weakstar} and 
\eqref{HuOmegaphi}, \eqref{RG}, we infer 
\[ H_{G(t\omega_\mu)}Q_j(t\omega_\mu)-\sigma_j \mathrm{e}^{i\alpha_j-it\sigma_j^2}Q_j(t\omega_\mu)\td_t,\infty 0\]
in $L^2_+(\R)$. Applying Proposition \ref{prop:cancel} with $\e(t)=0$, we conclude
$$\forall t\in \R,\quad  H_{G(t\omega_\mu)}Q_j(t\omega_\mu)-\sigma_j \mathrm{e}^{i\alpha_j-it\sigma_j^2}Q_j(t\omega_\mu)=0,$$
and setting $t=0$ completes the proof of Lemma \ref{spectralRmu}.  
\end{proof}

\subsection{Strong convergence in $H^{\frac 12}_+$}
Given $\e >0$, let us calculate  $\Vert v_{{\e}}\Vert_{H^{\frac 12}}^2$ where $$v_{{\e}}(t,x):= \sum_{{\mu\ge \e}} R_\mu(t,x).$$
First we check that, for $t$ large enough, the terms are almost orthogonal in $H^{1/2}$. Indeed, we recall from \eqref{RG} that
$$R_\mu \left (t,x+t\frac{\mu ^2}{2\pi}\right )=G_\mu (t\omega _\mu, x),$$
where $G_\mu (\theta,\cdot )$ belongs to a compact subset of $H^s_+(\R)$ for every $s$ in view of \eqref{eq:Gmu}. Therefore, for $\mu \ne \nu $,
$$\la R_\mu (t,\cdot ), R_\nu (t,\cdot )\ra _{H^{1/2}}=\left\la G_\mu \left (t\omega _\mu, \cdot -t\frac{\mu ^2-\nu ^2}{2\pi}\right ),G_\nu (t\omega _\mu, \cdot )\right \ra _{H^{1/2}}\td_t,\infty 0.$$
Consequently, we have 
$$\Vert v_N(t)\Vert_{H^{\frac 12}}^2= \sum_{N\mu\ge 1}  \Vert R_\mu(t)\Vert_{H^{\frac 12}}^2 +o(1),\quad t\to \infty.$$
As  $R_\mu$ is a  solution of the cubic Szeg\H{o} equation, one has 
$$\Vert R_\mu(t)\Vert_{L^2}^2=\Vert R_{\mu}^0\Vert_{L^2}^2,\quad \langle DR_\mu(t,\cdot ),R_\mu(t,\cdot )\rangle =\langle DR_\mu^0,R_\mu^0\rangle ,$$
and thus, in view of identities \eqref{RmuL2} and \eqref{RmuH1/2},  
 \begin{equation}\label{RmuH1/2bis}
 \Vert R_\mu(t)\Vert_{H^{\frac 12}}^2=\Vert \mathcal P_\mu u_0\Vert_{L^2}^2+2\pi \sum_{j\in J_\mu}\sigma_j^2.
 \end{equation}
We have
$$\Vert u(t)-v_{{\e}}(t)\Vert_{H^{\frac 12}}^2=\Vert u_0\Vert_{H^{\frac 12}}^2+\Vert v_{{\e}}(t)\Vert_{H^{\frac 12}}^2-2{\rm Re}\langle  u(t), v_{{\e}}(t)\rangle_{H^{\frac 12}}.$$
Then we observe that
  \begin{align*}\langle  u(t), v_{{\e}}(t)\rangle_{H^{\frac 12}}&=\sum_{\mu\ge {\e}}\langle u(t),R_\mu(t)\rangle_{H^{\frac 12}}\\
  &=\sum_{\mu \geq {\e}}\left \langle u\left (t,\cdot +t\frac{\mu ^2}{2\pi}\right ),R_\mu\left (t,\cdot +t\frac{\mu ^2}{2\pi}\right )\right \rangle_{H^{\frac 12}} ,
   \end{align*} and, using  \eqref{weakstar} and the strong compactness of $R_\mu\left (t,\cdot +t\frac{\mu ^2}{2\pi}\right )$, we conclude that
 \begin{align*}
 \lim_{t\to \infty}\Vert u(t)-v_{{\e}}(t)\Vert_{H^{\frac 12}}^2&=\Vert u_0\Vert_{H^{\frac 12}}^2 -\sum_{\mu \geq {\e}}\Vert R_\mu^0\Vert_{H^{\frac 12}}^2\\
 &=\sum_{\mu< {\e}}\left [\Vert \mathcal P_\mu u_0\Vert_{L^2}^2 +2\pi \sum_{j\in J_\mu}\sigma_j^2  \right ],
 \end{align*}
 which tends to $0$ as ${\e\to 0}$. This completes the proof of Theorem \ref{Main2}.
\subsection{Proof of Corollary \ref{solitonres}} Let us assume that $u_0$ is strongly generic. Then $N_\mu =1$ for every $\mu \in \mathcal M$. If $(\varphi_j)$ is an orthonormal basis of eigenfunctions of $H_{u_0}^2$ such that $\langle u_0,\varphi_j\rangle =\mu_j>0$, we have $E_{\mu_j}=\C \varphi_j$ and
$$\mathcal P_{\mu_j}f= \langle f,\varphi_j\rangle \varphi_j,\quad  \mathcal P_\mu L_{u_0}(t)\varphi_j =c_{\mu_j} t\varphi_j.$$
Consequently, in view of \eqref{R_mu}, we recover the expression \eqref{TravelingWaves}, \eqref{soliton}, of a soliton solution,
\begin{align*}
R_{\mu_j} (t,z)&=\frac{\mu_j \mathrm{e}^{-it\sigma_j^2}}{2i\pi}\frac{I_+(\varphi_j)}{\langle X^*\varphi_j,\varphi_j\rangle -z+c_{\mu_j}t}\\
&=\frac{i\mu_j^2}{2\pi\sigma_j}\frac{ \mathrm{e}^{-i(\alpha_j+t\sigma_j^2)}}{z-c_{\mu_j} t-\langle X^*\varphi_j,\varphi_j\rangle )}.
\end{align*}
The proof of Corollary \ref{solitonres} is complete.
\printbibliography
\end{document}